\documentclass{article}
\usepackage{graphicx} 
\usepackage{amsthm}
\usepackage{amsmath}
\usepackage{xcolor}

\allowdisplaybreaks

\numberwithin{equation}{section}

\newtheorem{lem}{Lemma}
\newtheorem{lemma}[lem]{Lemma}
\numberwithin{lem}{section}

\newtheorem{thm}{Theorem}
\newtheorem{theorem}[thm]{Theorem}
\numberwithin{thm}{section}

\numberwithin{cor}{section}

\newtheorem{defi}{Definition}
\newtheorem{definition}[defi]{Definition}
\numberwithin{defi}{section}

\newtheorem{rem}{Remark}
\newtheorem{remark}[rem]{Remark}
\numberwithin{rem}{section}

\title{ On the minimal length of addition chains}
\author{Jean-Marie De Koninck,  Nicolas Doyon,  William Verreault}
\date{}

\begin{document}

\maketitle

\begin{abstract}
We denote by $\ell(n)$ the minimal length of an addition chain leading to $n$ and we define the counting function $$
F(m,r):=\#\left\{n\in[2^m, 2^{m+1}):\ell(n)\le m+r\right\},
$$
where $m$ is a positive integer and $r\ge 0$ is a real number. We show that for $0< c<\log 2$ and for any $\varepsilon>0$, we have as $m\to \infty$,
$$
F\left(m,\frac{cm}{\log m}\right)<\exp\left(cm+\frac{\varepsilon m\log\log m}{\log m}\right)
$$
and
$$
F\left(m,\frac{cm}{\log m}\right)>\exp\left(cm-\frac{(1+\varepsilon)cm\log\log m}{\log m}\right).
$$
This extends a result of Erd\H{o}s which says that for almost all $n$, as $n\to\infty$,
$$
\ell(n)=\frac{\log n}{\log 2}+\left(1+o(1)\right)\frac{\log n}{\log \log n}.
$$
\end{abstract}
\vskip 10pt

\noindent


\vskip 10pt

\section{Introduction}
We say that $1=a_0<a_1<\cdots<a_k=n$ is an {\it addition chain} of length $k$ if $a_j=a_s+a_i$ with $0\le i,s<j$ for all $j,\, 1\le j\le k$. In other words, a term in an addition chain is obtained either by adding two distinct previous terms or by doubling a previous term.

For example, the following is an addition chain:
\begin{eqnarray*}
a_0&=&1,\\
a_1&=&a_0+a_0=1+1=2,\\
a_2&=&a_1+a_1=2+2=4,\\
a_3&=&a_2+a_0=4+1=5,\\
a_4&=&a_3+a_1=5+2=7,\\
a_5&=&a_3+a_2=5+4=9.
\end{eqnarray*}

A question of interest is: {\it Given a positive integer $n$, what is the shortest addition chain leading to $n$?} For example, if $n=10$, the following addition chain leads to $10$: $a_0=1, a_1=1+1=2, a_2=2+2=4, a_3=4+4=8, a_4=8+2=10.$  We can show that there exists no shorter addition chain leading to $10$.

When $n$ is a power of two, say $n=2^r$, the shortest addition chain leading to $n$ is of length $r$ and consists only of doubling steps, since we have the addition chain $a_0=1, a_1=2a_0=2,  a_2=2a_1=4,\ldots,  a_r=2a_{r-1}=2^r$.

Observe that the shortest addition chain leading to $n$ is not necessarily unique. For instance, we have $a_0=1, a_1=2, a_3=4, a_4=5$ as well as $b_0=1, b_1=2, b_2=3, b_4=5$ both being chains of minimal length leading to 5.

\begin{definition}
We denote by $\ell(n)$ the length of a shortest addition chain leading to $n$.
\end{definition}

In this note, we investigate a question related to the distribution of $\ell(n)$: {\it For a given positive integer $k$, how many integers $n$ can be obtained by an addition chain of length at most $k$?}   For example, if we take $k=3$, the following integers can be obtained by addition chains of length at most 3:
\begin{eqnarray*}
a_0=\mathbf{1},\\
a_0=1, a_1=\mathbf{2},\\
a_0=1, a_1=2, a_2=\mathbf{3},\\
a_0=1, a_1=2, a_2=\mathbf{4},\\
a_0=1, a_1=2, a_2=3, a_3=\mathbf{5},\\
a_0=1, a_1=2, a_2=3, a_3=\mathbf{6},\\
a_0=1, a_1=2, a_2=4, a_3=\mathbf{8}.
\end{eqnarray*}
It follows that $\mathbf{7}$ different integers can be obtained by addition chains of length at most $3$.

\begin{definition}
We denote by $H(k)$ the number of distinct integers that can be obtained by an addition chain of length at most $k$.  Explicitly,
$$
H(k):=\#\{n:\ell(n)\le k\}.
$$

\end{definition}

Given that the largest integer that can be obtained by an addition chain of length $k$ is $2^k$, we trivially have
$$
H(k)\le 2^k.
$$
It turns out that, in some sense, this simple upper bound is relatively close to the truth. The sequence $H(n)$ begins with
$$
H(1)=2,\, H(2)=4,\, H(3)=7,\,H(4)=12,\ldots
$$

Doubling steps, that is steps given by $a_r=2a_{r-1}$, play an important role in addition chains since for addition chains of minimal length, most of the steps must be doubling steps. It is therefore natural to work with intervals of the form $[2^m, 2^{m+1})$.
 Instead of working directly with the function $H$, we will investigate a closely related function.

\begin{definition}
Let $m$ be a positive integer and $r$ a positive real number. We define
$$
 F(m,r):=\#\left\{n\in [2^m,2^{m+1}):\ell(n)\le m+r\right\}.
 $$
\end{definition}
We will investigate the asymptotic behaviour of the function $F$. A critical region is given by
$$
r=\frac{cm}{\log m},
$$
where $0< c< \log 2$ is a fixed constant and $m\to \infty$.  As we will discuss in the next section, previous results imply that $F(m,r)=2^m$ when $c>\log 2$ and $m$ is large enough.

Relying on this notation, we will show
$$
F\left(m,\frac{cm}{\log m}\right)=\exp\left(cm(1+o(1))\right),
$$
as well as more precise results.

\section{Previous work}

Lower and upper bounds for the function $\ell(n)$ have been obtained.

Arnold Scholz \cite{scholz} seems to have been the first (in 1937) to study the length of an addition chain. In his paper, one can find what is now known as the {\it Scholz conjecture}, namely that
$$\ell(2^{n+1}-1)\le n +\ell(n+1).$$
This conjecture has been shown to be true for all $n\le 5\, 784\,688$ by Clift \cite{clift}. This conjecture has also been shown to be true for infinite families of integers, see \cite{bahig} and \cite{tall}.

Observe that it is easy to establish that
$$\ell(n) \le \log_2 n +\nu(n),$$
where $\log_2 m= \log m/\log 2$ and $\nu(n)$ is the number of digits equal to one in the binary expansion of $n$.
The best  result in the direction of the Scholz conjecture is the one obtained in 1939 by Arnold T. Brauer \cite{brauer}, namely the upper bound
$$\ell(2^{n+1}-1) \le n +\ell^*(n+1),$$
where  $\ell^*(n)$ is the minimum length of only those addition chains in which every $a_p$ can be written as $a_{p-1} + a_r$, $r\le p-1$. In fact, Brauer proved that
$$\ell(n) \le \min_{1\le r \le m} \left\{ \left( 1+ \frac 1r \right) \log_2 n +2^r -2 \right\},$$
where $2^m \le n <2^{m+1}$. In Brauer's upper bound, choosing $\displaystyle{ r = (1-\varepsilon) \frac{\log \log n}{\log 2}  }$, it follows that
for $n$ sufficiently large,
\begin{equation} \label{eq:erdos-3}
\ell(n) \le \log_2 n +(1+\varepsilon) \frac{\log n}{\log \log n}.
\end{equation}
In 1960, Erd\H{o}s \cite{erdos} showed that \eqref{eq:erdos-3} is best possible in the sense that for almost all $n$ (i.e., for all $n$ except a sequence of density 0),
$$\ell(n) = \log_2 n + (1+o(1)) \frac{\log n}{\log \log n}.$$

In 1975, Arnold Sch\"onhage \cite{schonhage} proved that
$$\ell(n)\ge \log_2 n + \log_2 \nu(n) -2.13.$$

To summarize, what has been proved so far unconditionally is that
$$\log_2 n + \log_2 \nu(n) -2.13 \le \ell(n) \le \lceil \log_2 n \rceil +\nu(n)-1.$$

Providing a more concrete bound, Thurber \cite{thurber2} showed that if $\nu(n)\ge 9$, then
$$
\ell(n)\ge \lfloor \log_2 n \rfloor+4.
$$

The problem of numerically computing $\ell(n)$ for a given $n$ is also an important challenge. Methods to effectively generate an addition chain of minimal length leading to a given integer can be found in \cite{thurber1}.

A broad view of the results regarding addition chains can be found in the book of Richard Guy \cite[p.169--171]{guy}.

\section{Result statement}

Our main results consist in an upper bound and a lower bound for the function $F$.

\begin{theorem}\label{t1}
For any $\varepsilon>0$ and for $0< c<\log 2$, we have for $m$ large enough
$$
F\left(m,\frac{cm}{\log m}\right)<\exp\left(cm+\frac{\varepsilon m\log\log m}{\log m}\right).
$$
\end{theorem}

\begin{theorem}\label{t2}
For any $\varepsilon>0$ and for $0< c<\log 2$, we have for $m$ large enough
$$
F\left(m,\frac{cm}{\log m}\right)>\exp\left(cm-\frac{c(1+\varepsilon)m\log\log m}{\log m}\right).
$$
\end{theorem}
These results partially answer a comment of Paul Erd\H{o}s \cite{erdos} who wrote that it would be interesting to obtain an asymptotic value for the distribution of the function $\ell(n)$.

\section{Proof of the upper bound}

The proof of the upper bound is the most complicated of the two and accordingly, it is divided into several subsections.

\subsection{Notation and definitions}

We begin the proof of the upper bound by providing a few more definitions. We denote by $\gamma$ the golden ratio
$$
\gamma:=\frac{1+\sqrt{5}}{2}\approx 1.618.
$$

\begin{definition}
For a fixed integer $m\ge 2$, we define
$$
\delta=\delta(m):=\frac{1}{\log m}.
$$
\end{definition}

 We expand upon the ideas of Erd\H{o}s as we will divide the addition steps into several types.

\begin{definition}
Let $m>4$ be a positive integer and $r$ a positive real number. Consider an addition chain consisting of $1=a_0<\cdots< a_{\lfloor m+r\rfloor }=n$ with $n\in [2^m, 2^{m+1})$.  Let the set ${\cal A}$ consist of the {\it doubling } steps of this chain.  We subdivide the nondoubling steps into three sets ${\cal B}$, ${\cal C}$ and ${\cal D}$ according to their associated  relative growth rate:
\begin{eqnarray*}
{\cal A}:&=&\{1\le j\le m+r: a_j=2a_{j-1}\},\text{ the doubling steps,}\\
{\cal B}:&=&\{1\le j\le m+r: \gamma a_{j-1}\le a_j<2a_{j-1}\}, \text{ the large steps,}\\
{\cal C}:&=&\{1\le j\le m+r:(1+\delta)a_{j-1}\le a_{j}<\gamma a_{j-1}\}, \text{ the midsize steps,}\\
{\cal D}:&=&\{1\le j\le m+r:a_j<(1+\delta)a_{j-1}\}, \text{ the small steps}.
\end{eqnarray*}

 We denote their respective cardinalities by
$$
A:=\#{\cal A},\, B:=\#{\cal B},\, C:=\#{\cal C}, \, D:=\#{\cal D}.
$$
\end{definition}

\subsection{\bf Remarks on the size of the sets ${\cal A}$, ${\cal B}$, ${\cal C}$ and ${\cal D}$}

We now obtain bounds on $A$, $B$, $C$ and $D$.

\begin{lemma}\label{l1}
Suppose that an integer $n \in [ 2^m, 2^{m+1})$ is obtained by an addition chain of length $\lfloor m+r \rfloor $ such that the number of steps of each type is given by $A$, $B$, $C$ and $D$, respectively. Then
$$
B+C+D\le r \frac{1}{1-\log_2\gamma}.
$$
\end{lemma}

\begin{proof}
In \cite{schonhage}, it was shown by induction that
\begin{equation}\label{e1}
n\le 2^A \gamma^{B+C+D}.
\end{equation}
 The intuition being that a chain with additive (nondoubling) steps cannot grow faster than the Fibonacci sequence.

Using the fact that $A+B+C+D=\lfloor m+r\rfloor$ and  $n\ge 2^m,$ we obtain from (\ref{e1}) that
$$
2^m\le 2^{m+r-(B+C+D)}\gamma^{B+C+D},
$$
which yields
$$
1\le 2^r \left(\frac{\gamma}{2}\right)^{B+C+D}.
$$
From this, we get
$$
B+C+D\le r\frac{1}{1-\log_2 \gamma},
$$
completing the proof of Lemma \ref{l1}.
\end{proof}
We also have the following upper bound on the size of the set ${\cal D}$.
\begin{lemma}\label{l2}
Suppose that an integer $n \in [2^m,  2^{m+1})$ is obtained by an addition chain of length $\lfloor m+r\rfloor $ such that the number of steps of each type is given by $A$, $B$, $C$ and $D$, respectively. We have
$$
D\le \frac{r -C(1-\log_2\gamma)}{1-\log_2(1+\delta)}.
$$
\end{lemma}

\begin{proof}
Using the definition of the sets ${\cal A}$, ${\cal B}$, ${\cal C}$ and ${\cal D}$, we  have
\begin{eqnarray*}
2^m&\le& n \le 2^{A+B}\gamma^C (1+\delta)^{D}\\
&\le &2^{m+r-C-D}\gamma^C(1+\delta)^D.
\end{eqnarray*}
From this, we get
$$
1\le 2^r\left(\frac{\gamma}{2}\right)^C\left(\frac{1+\delta}{2}\right)^D.
$$
Taking the base 2 logarithm of each side of the inequality yields
$$
0\le r -C(1-\log_2\gamma)-D(1-\log_2(1+\delta)),
$$
which we can rearrange to conclude
$$
D\le \frac{r -C(1-\log_2\gamma)}{1-\log_2(1+\delta)},
$$
thus proving Lemma \ref{l2}.
\end{proof}

\subsection{The number of ways of choosing the sets ${\cal A}$, ${\cal B}$, ${\cal C}$ and$\,{\cal D}$}

The sets ${\cal A}$, ${\cal B}$, ${\cal C}$ and ${\cal D}$ specify the {\it position} in the addition chain of the steps of a given type.

We define
$$
S=S(A,B,C,D,m,r)
$$
as the number of ways of choosing these sets for given values of $A, B, C, D, m$ and $r$.

It is trivial that
$$
S\le {\lfloor m+r\rfloor \choose A, B, C, D}
$$
where we use the multinomial coefficient notation:
$$
{\lfloor m+r\rfloor \choose A, B, C,D}=\frac{\lfloor m+r\rfloor !}{A!B!C!D!}
$$
with $\lfloor m+r\rfloor=A+B+C+D$.

We can however improve upon this by observing the following result.
\begin{lemma}\label{l3}
A step in ${\cal B}$ must be preceded by a step in either ${\cal C}$ or ${\cal D}$.
\end{lemma}
\begin{proof}
Suppose that $j\in {\cal B}$. Since a step in ${\cal B}$ is not a doubling step, we have
$$
a_j\le a_{j-1}+a_{j-2}.
$$
By definition of the set ${\cal B}$, we also have
$$
a_j\geq \gamma a_{j-1},
$$
and since $a_j, a_{j-1}$ are integers while $\gamma$ is irrational, it must be that $a_j> \gamma a_{j-1}$.
Thus
$$
\gamma a_{j-1}< a_{j-1}+a_{j-2},
$$
from which we conclude
$$
a_{j-1}<\frac{1}{\gamma-1} a_{j-2}=\gamma a_{j-2}
$$
and
$$
j-1\in {\cal C}\cup {\cal D}. \qedhere
$$
\end{proof}
From the above, we can conclude the following bound.
\begin{lemma}\label{l4} Let $m$ be a positive integer and $r$ a positive real number such that $r=cm/\log m$ with $c$ a fixed constant, $0<c< \log 2$.  We have as $m\to \infty$
$$
S\le {\lfloor m+r\rfloor \choose C+D}\exp(O(r)).
$$
\end{lemma}
\begin{proof}
The number of ways of choosing the sets ${\cal C}$ and ${\cal D}$ is given by
$$
{\lfloor m+r\rfloor \choose C+D}{C+D\choose D}.
$$
By Lemma \ref{l3}, the number of ways to choose the set ${\cal B}$ is less  than or equal to
$$
{B+C+D-1\choose B-1}.
$$
We have thus shown
$$
S\le {\lfloor m+r\rfloor \choose C+D}{C+D\choose D}{B+C+D-1\choose B-1}.
$$
This implies
$$
S\le {\lfloor m+r\rfloor \choose C+D}2^{C+D}2^{B+C+D}.
$$
Given that from Lemma \ref{l1}, $B+C+D=O(r)$, we conclude
$$
S\le {\lfloor m+r\rfloor \choose C+D}\exp(O(r)).
$$

\end{proof}

We will also make use of the following lemma.

\begin{lemma}\label{l4a}
Assume as above that $m$ is a positive integer and that $r$ is a positive real number with $r=c m/\log m$ and $c$ a fixed constant such that $0<c <\log 2$. Further assume that the addition steps are grouped in $k$ sequences of consecutive addition steps. Then
$$
S\le {\lfloor m+r\rfloor \choose k}\exp(O(r))
$$
as $m\to \infty$.
\end{lemma}
\begin{proof}
Suppose that the steps in ${\cal B}\cup{\cal C}\cup{\cal D}$ are grouped into $k$ blocks and that the length of each block is respectively $b_1, b_2,\ldots, b_k$.  We then have
$$
b_1+b_2+\cdots+b_k=B+C+D.
$$
The number of ways of choosing the integers $b_1,\, b_2,\ldots, b_k$ is less than or equal to
$$
2^{B+C+D}=\exp(O(r)).
$$
The number of ways of choosing if the addition steps belong to ${\cal B}$, ${\cal C}$ or ${\cal D}$ is less than or equal to
$$
3^{B+C+D}=\exp(O(r)).
$$
It remains to specify how the $A$ doubling steps will be distributed amongst the $k$ additive blocks.  This is equivalent to choosing a sequence of integers $t_1,\, t_2,\ldots,\,t_{k+1}$ in such a way that
$$
t_1+t_2+\cdots+t_{k+1}=A
$$
with $t_1\ge 0,\, t_{k+1}\ge 0$ and $t_j>0,\, 2\le j\le k$.
The number of ways of doing so is less than or equal to
$$
{A+k\choose k}={\lfloor m+r\rfloor \choose k},
$$
thus completing the proof of Lemma \ref{l4a}.

\end{proof}

\subsection{The number of ways of choosing the added elements at each step}

To completely specify an addition chain, we still have to choose which elements will be added at each step.

First observe that for doubling steps in ${\cal A}$ there is no choice to make, since if $j\in{\cal A}$ then necessarily $a_j=2a_{j-1}$. We start by proving the following result.

\begin{lemma}\label{l5}
Let $m$ be a positive integer and $r$ a positive real number such that $r=cm/\log m$ with $c$ a fixed constant, $0<c< \log 2$. Let $1=a_0<\ldots< a_{\lfloor m+r \rfloor}=n$ be an addition chain of length $\lfloor m+r\rfloor$ leading to $n$ with $n\in [2^m, 2^{m+1})$.  Assume that the sets ${\cal A}$, ${\cal B}$, ${\cal C}$ and ${\cal D}$ are fixed.
Denote by $R$ the number of ways of choosing which integers will be added in steps belonging to ${\cal B}$. Then,
$$
R=\exp(O(r))
$$
as $m\to \infty$.
\end{lemma}

\begin{proof}
Suppose that $j\in {\cal B}$ and that
$$
a_j=a_s+a_k,\, j>s\ge k,\, k<j-1.
$$
We have
$$
\gamma a_{j-1}\le a_j=a_s+a_k\le a_{j-1}+a_k,
$$
from which we deduce
$$
a_k\ge (\gamma-1)a_{j-1}.
$$
It follows that all $u$ such that $k<u<j$  are in ${\cal C}\cup{\cal D}$ as assuming the contrary would yield $a_{j-1}>a_k$.  The number of ways of choosing $s$ and $k$ is thus less than or equal to $h(j)^2$ where  for $j\in {\cal B}$,
$$
h(j):=\max_{t\ge 1}\{j-u\in {\cal C}\cup{\cal D} \text{ for all } 1\le u\le t \}.
$$
In other words, $h(j)$ is the maximal integer $t$ such that all the integers in the interval $[j-t, j-1]$ are in the set ${\cal C}\cup{\cal D}$.
We thus have
\begin{equation}\label{r1}
R\le \prod_{j\in {\cal B}}h(j)^2.
\end{equation}
From Lemma \ref{l3}, we have
\begin{equation}\label{r2}
\sum_{j\in {\cal B}}h(j)\le C+D.
\end{equation}
From (\ref{r2}) and the arithmetic-geometric mean inequality, we have
$$
\prod_{j\in {\cal B}} h(j)^{1/B}\le \frac{1}{B}\sum_{j\in {\cal B}}h(j)\le \frac{C+D}{B}.
$$
From (\ref{r1}), we can conclude
$$
R\le\left(\frac{C+D}{B}\right)^{2B}=\exp(O(r)). \qedhere
$$
\end{proof}

We will now prove an upper bound for the number of ways of choosing the integers which will be added in steps belonging to ${\cal C}$. This turns out to be one of the most technical steps in the proof.

\begin{lemma}\label{l6}
Let $m$ be a positive integer and $r$ a positive real number such that $r=cm/\log m$ with $c$ a real constant such that $0<c< \log 2$. Let $1=a_0<a_1<\cdots <a_{\lfloor m+r\rfloor}=n$ with $n\in [2^m, 2^{m+1})$ be an addition chain of length $\lfloor m+r\rfloor$. Assume that the sets ${\cal A}$,  ${\cal B}$, ${\cal C}$ and ${\cal D}$ are fixed.
Denote by $T$ the number of ways of choosing which integers will be added in steps belonging to ${\cal C}$. For any $\varepsilon>0$, we have as $m\to \infty$
$$
T\le (\log m)^{(1+\varepsilon)2C} \exp(O(r)).
$$
\end{lemma}

\begin{proof}
The intuitive idea behind the proof of Lemma \ref{l6} is that for steps $j$ in ${\cal C}$, the number of choices will be limited since if $a_j=a_k+a_s$, then neither $a_k$ nor $a_s$ can be very small (smaller than $\delta a_j$).

We denote by $\beta_1,\ldots, \beta_{C}$ the elements of the set ${\cal C}$ in increasing order. For $1\le s\le C$, we write
$$
a_{\beta_s}=a_{\beta'_s}+a_{\beta''_s},
$$
where $\beta'_s$ and $\beta''_s$ are not necessarily in ${\cal C}$.

We will now get an upper bound for the number of ways of choosing the values of $\beta'_s$ and $\beta''_s$.

Given that $\beta_s\in {\cal C}$, we have $a_{\beta_s}\ge (1+\delta) a_{\beta_s-1}$. It follows that $a_{\beta'_s},a_{\beta''_s}\le a_{\beta_s-1}\le (1+\delta)^{-1}a_{\beta_s}$.  From this and from $a_{\beta'_s}+a_{\beta''_s}=a_{\beta_s}$, we get
$$
a_{\beta'_s}=a_{\beta_s}-a_{\beta''_s}\ge a_{\beta_s}-\frac{1}{1+\delta}a_{\beta_s}=  a_{\beta_s}-(1-\delta+\delta^2-\cdots)a_{\beta_s}\ge \delta a_{\beta_s}.
$$
Similarilly, we have $a_{\beta''_s}\ge \delta a_{\beta_s}$.

For $1\le s\le C$, let
$$
u(s):=\#\{j: \delta a_{\beta_s}\le a_j\le a_{\beta_s},\, j\in {\cal A}\cup {\cal B}\cup {\cal C}\}.
$$
Then
$$
(1+\delta)^{u(s)}\le \frac{1}{\delta},
$$
which yields
\begin{equation}\label{eee1}
u(s)\le \frac{-\log \delta }{\log(1+\delta)}\le \frac{-2\log \delta}{\delta}
\end{equation}
provided that $\delta$ is small enough (which is the case if $m$ is large enough). Let $t=t(s)$ be the smallest integer such that
$$
\#\{j: t(s)\le j\le \beta_s,\, j\in {\cal A}\cup {\cal B}\cup {\cal C}\}\le \frac{-2\log \delta}{\delta}.
$$
The interest of this is that $\beta'_s$ and $\beta''_s$ must be chosen in the interval $[t(s), \beta_s]$.  Therefore, the number of integers in this interval is at most $\displaystyle \frac{-2\log \delta}{\delta}+\eta(s)$, where
$$
\eta(s):=\#\{j: t(s)\le j \le \beta_s, j\in {\cal D}\}.
$$
Hence, the number of ways of choosing the integers $\beta'_s$ and $\beta''_s$ for $1\le s\le C$ is less than or equal to
$$
\prod_{s=1}^{C}\left(\frac{-2\log\delta}{\delta}+ \eta(s)\right)^2.
$$
We have
$$
\sum_{s=1}^{C} \left(\frac{-2\log \delta}{\delta}+\eta(s)\right)= \frac{-2C\log \delta}{\delta} +\sum_{s=1}^{C}\eta(s).
$$
Because each element in ${\cal D}$ is counted in at most $\frac{-2\log \delta}{\delta}$ sets of the form $\{j: t(s)\le j \le \beta_s, j\in {\cal D}\}$, we have
$$
\sum_{s=1}^C\eta(s)\le \frac{-2D\log \delta}{\delta}.
$$
We can conclude
$$
\sum_{s=1}^{C}\left(\frac{-2\log \delta}{\delta}+\eta(s)\right)\le \frac{-2(C+D)\log \delta}{\delta}.
$$
By the arithmetic-geometric mean inequality, we get
\begin{eqnarray*}
\prod_{s=1}^{C} \left(\frac{-2\log \delta}{\delta}+\eta(s)\right)^{1/C}&\le&\frac{1}{C}\sum_{s=1}^{C}\left(\frac{-2\log \delta}{\delta}+\eta(s)\right)\\
&\le& \frac{-2(C+D)\log \delta}{C\delta}.
\end{eqnarray*}
Raising both sides to the power $2C$, it follows that
\begin{equation}\label{ub1}
\prod_{s=1}^{C} \left(\frac{-2\log \delta}{\delta}+\eta(s)\right)^2\le\left(\frac{-2\log \delta}{\delta} \right)^{2C}\left(\frac{C+D}{C}\right)^{2C}.
\end{equation}
Writing $C=\xi D$, we get
\begin{eqnarray*}
\left(\frac{C+D}{C}\right)^{2C}&=&\left(\frac{\xi+1}{\xi}\right)^{2\xi D}\\
&=&\exp\left(2\xi D\log(1+1/\xi)\right)\\
&\le& \exp(2D).
\end{eqnarray*}
From Lemma \ref{l1}, it follows that
$$
\left(\frac{C+D}{C}\right)^{2C}=\exp(O(r)).
$$
From the definition of $\delta$, we have for any $\varepsilon>0$ and for $m$ large enough,
$$
\left(\frac{-2\log \delta}{\delta}\right)^{2C}=(2\log m\log\log m)^{2C}\le (\log m)^{(1+\varepsilon)2C}.
$$
We can thus conclude
$$
T\le (\log m)^{(1+\varepsilon)2C} \exp(O(r)). \qedhere
$$
\end{proof}

Finally, we provide an upper bound for the number of ways of choosing which integers need to be added in steps belonging to ${\cal D}.$
\begin{lemma}\label{l7}
Let $m$ be a positive integer and let $r$ be a positive real number such that $r=cm/\log m$ with $c$ a fixed constant, $0<c<\log 2$. Let $1=a_0<a_1<\cdots<a_{\lfloor m+r\rfloor}=n$ be an addition chain of length $\lfloor m+r\rfloor $ leading to $n$ with $n\in [2^m, 2^{m+1})$. Assume that the sets ${\cal A}$, ${\cal B}$, ${\cal C}$  and ${\cal D}$ are fixed.
The number of ways of choosing the integers added in steps belonging to ${\cal D}$ is less than or equal to
$$
\exp\left(D\log m+(1+o(1))D\log\log m\right)
$$
as $m\to \infty$.
\end{lemma}
\begin{proof}
We follow the approach used by Erd\H{o}s in  \cite{erdos}.  Assume that $j\in {\cal D}$ with $a_j=a_s+a_k$. The pair $(s,k)$ must be chosen in the box $[1,m+r]\times [1,m+r]$ which contains $\lfloor m+r\rfloor^2$ elements. Thus, the number of ways of choosing which integers are added in steps of ${\cal D}$ is less than or equal to
$$
{\lfloor m+r\rfloor^2\choose D}.
$$
From this, we can prove Lemma \ref{l7} by observing
$$
{\lfloor m+r\rfloor^2\choose D}\le \left(\frac{e(m+r)^2}{D}\right)^D\le\exp\left(D\log m+(1+o(1))D\log\log m\right). \qedhere
$$
\end{proof}

\subsection{Some intermediate results}

We now have the tools to obtain some intermediate results that will take us closer to the proof of the upper bound.

\begin{definition}
 Let $m$ be a positive integer and $r$ a positive real number. We denote by
$$
 N=N(A,B,C,D,m,r)
 $$
the number of distinct addition chains of length $\lfloor m+r\rfloor$ leading to an integer $n \in [2^m, 2^{m+1})$ with  $A,B,C,D$ fixed.
\end{definition}
The results of the previous subsections will yield an upper bound on $N$.

\begin{lemma}\label{l8} Let $m$ a positive integer and let $r$ be a positive real number with $r=cm/\log m$ where $c$ is a fixed constant with $0<c< \log 2$.  We have
$$
N(A,B,C,D,m,r)\le \exp\left(D\log m+(1+o(1))(2D+3C)\log\log m\right)
$$
as $m\to \infty$.
\end{lemma}
\begin{proof}
The proof of Lemma \ref{l8} follows directly from Lemmas \ref{l4}, \ref{l5}, \ref{l6} and  \ref{l7}.
\end{proof}

From Lemmas \ref{l2} and \ref{l8}, we have
\begin{eqnarray}
N
&\le&\exp\left(\frac{r -C(1-\log_2\gamma)}{1-\log_2(1+\delta)}\log m +(1+o(1))\left(2\frac{r-C(1-\log_2\gamma)}{1-\log_2(1+\delta)}+3C\right)\log\log m\right) \nonumber\\
&\le& \exp\left(\frac{r -C(1-\log_2\gamma)}{1-\log_2(1+\delta)}\log m +(1+o(1))\left(2r+3C\right)\log\log m\right)\label{eqnew1}.
\end{eqnarray}
From (\ref{eqnew1}), we can obtain the following upper bound which, though weaker than Theorem \ref{t1}, will be useful.
\begin{lemma}\label{l9}
Let $m$ be a positive integer and $0<c< \log 2$ a real number.Then, for any $\varepsilon>0$ and for $m$ large enough,
$$
F\left(m,\frac{cm}{\log m}\right)\le \exp\left(cm+\frac{(2+\varepsilon)cm\log\log m}{\log m}\right).
$$
\end{lemma}
\begin{proof}
Assume that $r=cm/\log m$ with $0<c<\log 2$. From inequality (\ref{eqnew1}), we have as $m\to \infty$,
$$
N(A,B,C,D,m,r)\le\exp\left(r\log m+(1+o(1))2r\log\log m \right).
$$
From this, we conclude that
\begin{eqnarray*}
F(m,r)&\le& \sum_{A,B,C,D\le m+r}N(A,B,C,D,m,r)\\
&\le& (m+r)^4\max_{A,B,C,D\le m+r}N(A,B,C,D,m,r)\\
&\le&\exp\left(r\log m+(1+o(1))2r\log\log m \right).
\end{eqnarray*}
Using $r=cm/\log m$ completes the proof of Lemma \ref{l9}.
\end{proof}
It is interesting to note that from Lemma \ref{l9} and inequality (\ref{eqnew1}), we can conclude that addition chains with a large value of $C$, for instance chains such that
\begin{equation}\label{Cub}
C<\frac{7r\log \log m}{(1-\log_2\gamma)\log m},
\end{equation}
have a negligible contribution to $F(m,r)$.

\subsection{Addition blocks}
We divide the addition steps (that is, steps in ${\cal B}\cup{\cal C}\cup{\cal D}$) in blocks of consecutive addition steps.  Say we have $K$ blocks of lengths
$$
L_1,L_2,\,  \ldots,L_K
$$
with
$$
L_1+L_2+\cdots+L_k=B+C+D.
$$
Suppose that
$$
a_j,\ldots, a_{j+L-1}
$$
is a sequence of $L>1$ consecutive addition steps. That is, suppose that $s\in {\cal B}\cup{\cal C}\cup{\cal D}$ for $j\le s\le j+L-1$ and that $a_{j-1}\in{\cal A}$, $a_{j+L}\in{\cal A}$.

 Then, after $a_{j-1}=2a_{j-2}$, the sequence begins with
$$
a_j=a_{j-1}+a_t,\, t<j-1.
$$
Indeed, if $a_j=a_s+a_t$ with $s,t<j-1$, we would have $a_j \le a_{j-2}+a_{j-2}=a_{j-1}$, a contradiction. For $s$ such that $j\le s\le j+L-1$, if we write
$$
a_s=a_v+a_t,\, v\ge t,
$$
then $v\ge j-1$ or else $a_s\le a_{j-1}$. Indeed, assuming the contrary yields
$$
a_s=a_v+a_t\le a_{j-2}+a_{j-2}=a_{j-1},
$$
a contradiction.

We will use this observation and the block structure to refine the upper bound on the number of ways of choosing the integers that are in steps belonging to ${\cal D}$.
\begin{lemma}\label{l10}
Consider a block of $L$ consecutive additive steps with $L'$ steps in ${\cal D}$ and $L''$ steps in ${\cal B}\cup {\cal C}$ in an addition chain of length $\lfloor m+r\rfloor$. Then, the number of ways of choosing which integers are added in steps of type ${\cal D}$ of this block is less than or equal to
$$
{(L'+L'')\lfloor m+r\rfloor \choose L'}.
$$
\end{lemma}
\begin{proof}
Suppose that the addition block ranges from $a_j$ to $a_{j+L-1}$, that $j\le s\le j+L-1$ and that $a_s=a_v+a_t$ with $t\le v$.  The pair $(t,v)$ must be chosen in the set $[1, m+r]\times [j-1,j+L-2]$. Notice that we use the fact that $v\in [j-1,j+L-2]$ but don't use the condition $t\le v$ as it would not lead to significant improvements. The cardinality of the set $[1, m+r]\times [j-1,j+L-2]$ is $L\lfloor m+r\rfloor=(L'+L'')\lfloor m+r\rfloor$.  Since we must choose $L'$ distinct pairs $(t,v)$, the number of ways of making this choice is less than or equal to
$$
{(L'+L'')\lfloor m+r\rfloor \choose L'}. \qedhere
$$
\end{proof}
Using Lemma \ref{l10}, we can improve the bound obtained in Lemma \ref{l7}.
\begin{lemma}\label{l11}
Let $m$ be a positive integer and $r$ a positive real number such that $r=cm/\log m$ with $c$ a fixed constant, $0<c< \log 2$. Consider $1=a_0<a_1<\cdots<a_{\lfloor m+r\rfloor}$ an addition chain of length $\lfloor m+r\rfloor$ leading to $n$ with $n\in [2^m,  2^{m+1})$.  Assume that the sets ${\cal A}$, ${\cal B}$, ${\cal C}$ and ${\cal D}$ are fixed.  The number of ways of choosing which integers are added in steps belonging to ${\cal D}$ is less than or equal to
$$
\exp(D\log m+O(m/\log m)).
$$
\end{lemma}
\begin{proof} Assume that the addition steps are divided in $k$ blocks of lengths $L_1,\ldots, L_k$.  Assume that the number of steps in ${\cal D}$ in the block $s$, with $1\le s\le k$, is $L_s'$ and that the number of steps in ${\cal B}\cup {\cal C}$ in this block is $L_s''$. From Lemma \ref{l10}, the number of ways of choosing the integers added in steps belonging to ${\cal D}$ is less than or equal to
\begin{eqnarray*}
\prod_{1\le s\le k} {(L_s'+L_s'')\lfloor m+r\rfloor )\choose L_s'}&\le& \prod_{1\le s\le k} \left(\frac{e(L_s'+L_s'')(m+r)}{L_s'}\right)^{L_s'}\\
&=&(e(m+r))^D\prod_{1\le s\le k}\left(1+\frac{L_s''}{L_s'}\right)^{L_s'}\\
&\le& (e(m+r))^D\exp\left(\sum_{s=1}^k L_s''\right)\\
&=&(e(m+r))^De^{B+C}\\
&=&\exp(D\log m+O(m/\log m)).
\end{eqnarray*}
For the last equality, we used the fact that $B+C=O(r)$ and $r=O(m/\log m)$.

\end{proof}

Using Lemma \ref{l11}, we can improve Lemma \ref{l9} and obtain the following result.
\begin{lemma}\label{l12}
Let $m$ be a positive integer and $c,\, 0<c< \log 2$, be a real constant.  For any $\varepsilon>0$, we have for $m$ large enough
$$
F\left(m,\frac{cm}{\log m}\right)\le \exp\left(cm+\frac{(1+\varepsilon)cm\log\log m}{\log m}\right).
$$
\end{lemma}
While this result is still weaker than Theorem \ref{t1},  we are getting closer.

\begin{proof}
The proof of Lemma \ref{l12} follows rather directly from our previous results. The number of ways of choosing $A$, $B$, $C$ and $D$ is $O(m^4)$.   By Lemma $\ref{l4}$, the number of ways of choosing the sets ${\cal A}$, ${\cal B}$,  ${\cal C}$ and ${\cal D}$ is  less than or equal to
$$
{\lfloor m+r\rfloor \choose C+D}\exp(O(r)).
$$
By Lemmas \ref{l5}  and  \ref{l6}, for any $\varepsilon>0$ and for $m$ large enough, the number of ways of choosing the integers added in steps belonging to  ${\cal B}\cup {\cal C}$ is less than or equal to
$$
(\log m)^{(1+\varepsilon)2C}\exp\left(O(r)\right).
$$
Using Lemma \ref{l11} and assuming $\displaystyle r=\frac{cm}{\log m}$, we obtain
\begin{eqnarray}
F\left(m,\frac{cm}{\log m}\right)&\le& {\lfloor m+r\rfloor\choose C+D}\exp(D\log m)\exp(O(r))\nonumber\\
&=&{\lfloor m+r\rfloor \choose C+D}\exp(D\log m)\exp\left(O\left(\frac{m}{\log m}\right)\right)\label{nv1}.
\end{eqnarray}
Finally, Lemma \ref{l12} follows from (\ref{nv1}) and Lemma \ref{l2}.
\end{proof}

We end this subsection by observing that from the above results, the number of addition chains of length $\lfloor m+r\rfloor$ leading to $n$ with $n\in [2^{m}, 2^{m+1})$ and  with $k$ addition blocks is less than or equal to
$$
{\lfloor m+r\rfloor \choose k}\exp\left(D\log m+o\left(\frac{m\log\log m}{\log m}\right)\right).
$$
This implies that if $k=o(m/\log m)$, the proof of Theorem \ref{t1} is complete.  In the rest of the demonstration, we will thus assume that there is a positive constant $c_1$ such that $k>c_1 m/\log m.$

\subsection{Minimal addition chains}
As discussed in the Introduction, several addition chains of the same length can lead to the same integer. Thus far, we did not distinguish between such chains.  To make further progress, we will need a few additional definitions.

\begin{definition}
Let $m,n$ be positive integers and let $r$ be a positive real number.  Let
$$
1=a_0<a_1<\cdots<a_{\lfloor m+r\rfloor}=n
$$
and
$$
1=b_0<b_1<\cdots<b_{\lfloor m+r\rfloor}=n
$$
be two addition chains of the same length leading to the same integer.

Assume that $b_j\le a_j$ for all $j$.  Assume further that there exists an $s$, $1\le s\le m+r$, such that $b_s<a_s$.

In this case, we say that the chain $a_0,\ldots, a_{\lfloor m+r\rfloor}$ {\bf dominates} the addition chain $b_0,\ldots, b_{\lfloor m+r\rfloor}$.
\end{definition}

If an addition chain dominates another addition chain of the same length and leads to the same integer, we say that the first chain is {\bf not valid}.  Our goal is to count only valid addition chains.

For example, consider the two addition chains $a_0=1, a_1=2,a_3=4,a_4=5$ and $b_0=1, b_1=2,b_3=3,b_4=5$. The chain $a_j,\, 0\le j\le 4$, dominates the chain $b_j,\, 0\le j\le 4$, and thus is not valid.

We also define the notion of a {\it marked} element of an addition chain.
\begin{definition}
Let $1=a_0<a_1<\cdots<a_{\lfloor m+r\rfloor }=n$ be an addition chain.  We say that the step $s$ is {\bf marked} if there exists a $j>s+1$ such that $a_j=a_s+a_t$ for some $t$ or else the addition chain cannot be a valid chain of minimal length.
\end{definition}
We provide a few examples to illustrate this notion.
Suppose that an addition chain begins with
$$
a_0=1, a_1=2, a_2=4, a_3=5, a_4=8, \ldots
$$
The step 3 ($a_3=5$) is {\bf marked} since if it is not used in a later step of the chain, it can simply be removed from the addition chain which would yield a chain of smaller length leading to the same integer.

We now give another example that will be more relevant in the rest of the proof.  Suppose that an addition chain begins with
$$
a_0=1, a_1=2, a_2=4, a_3=8, a_4=10, a_5=20,\ldots
$$
then the step 3 ($a_3=8$) is marked since if it is not used in a future step of the chain, the chain could be replaced by
$$
b_0=1, b_1=2, b_2=4, b_3=5, b_4=10, b_5=20,\, b_j=a_j, \, j>5.
$$
Since the chain $a_j$ dominates the chain $b_j$, it is not valid.
The previous example can be generalized.  Assume that an addition chain contains the sequence
$$
a_0=1,\ldots, a_s=2a_{s-1},\ldots, a_j=2a_{j-1},\, a_{j+1}=a_j+a_s,\ldots,
$$
then the step $a_j$ is marked. Indeed if $a_j$ is not used  in a step following $a_{j+1}$, we can replace the addition chain $a_0, a_1,\ldots, a_{\lfloor m+r\rfloor}$, by the chain $b_0, b_1,\ldots, b_{\lfloor m+r\rfloor}$, defined as $b_t=a_t,\, 0\le t\le j-1$,  $b_j=b_{j-1}+b_{s-1}$, $b_{j+1}=2b_j$ as well as $b_t=a_t,\, t>j+1$.  The important point is to observe that
$$
b_{j+1}=2b_j=2b_{j-1}+2b_{s-1}=2a_{j-1}+2a_{s-1}=a_{j+1}.
$$

We extend the definition of a marked step of a chain to a marked addition block.
\begin{definition}
Let $1=a_0<a_1<\cdots<a_{m+r}$ be an addition chain.  Assume that this chain contains the sequence
$$
a_j=2a_{j-1}<a_{j+1}<\cdots<a_{j+L}<a_{j+L+1}=2a_{j+L},
$$
where $a_{j+1},\ldots,a_{j+L}\in {\cal B}\cup {\cal C}\cup {\cal D}$.  We say that the addition block $a_{j+1},\ldots a_{j+L}$ is {\bf marked} if at least one element preceding an element of this block must be used in a subsequent block for the chain to be valid.
\end{definition}

We illustrate this definition with an example.  Consider an addition chain that begins with the sequence
\begin{eqnarray*}
& &a_0=1, a_1=2a_0=2, a_2=2a_1=4, a_3=2a_2=8, a_4=a_3+a_1=10, \\
& &a_5=a_4+a_2=14, a_6=2a_5=28,\ldots
\end{eqnarray*}
In this example, the addition block $a_4,a_5$ is marked.  Indeed, assume that neither $a_3$ nor $a_4$ is used in a subsequent addition block, then the beginning of the chain can be replaced by
\begin{eqnarray*}
& &b_0=1, b_1=2b_0=2, b_2=2b_1=4, b_3=b_2+b_0=5, b_4=b_3+b_1=7, \\
& &b_5=2b_4=14, b_6=2b_5=28,\ldots
\end{eqnarray*}
This example can be generalized into the following result.

\begin{lemma}\label{l13}
Let $1=a_0<a_1<\cdots<a_{\lfloor m+r\rfloor}$ be an addition chain.  Let $a_{j+1}, \ldots,a_{j+L}$ be an addition block of this chain.  For $j+1\le s\le j+L$, write
$$
a_s=a_{t(s)}+a_{r(s)},\, t(s)\ge r(s).
$$
Assume that for all $s$ with $j+1\le s\le j+L$,
$$
t(s),r(s)<j+1\Rightarrow t(s),r(s) \in {\cal A},
$$
then the addition block $a_{j+1}, \ldots,a_L$ is {\bf marked}.

\end{lemma}
\begin{proof}
Let $1=a_0<a_1<\cdots < a_{\lfloor m+r\rfloor}$ be an addition chain. Let $a_{j+1},\ldots, a_{j+L}$ be an addition block of this chain. Assume that none of the steps $a_j,\ldots, a_{j+L-1}$ are used in subsequent addition blocks.  Then we can show that this addition chain is not valid by comparing it to the addition chain  $b_0,\ldots, b_{\lfloor m+r\rfloor}$ defined by
$$
b_s=a_s,\, s<j,\, b_s=a_{s+1}/2,\, j\le s\le j+L-1,\, b_s=a_s,\, s\ge j+L. \qedhere
$$
\end{proof}

For any addition chain, we divide the addition blocks into two types.
\begin{definition}
We say that an addition block is of {\bf type 1} if either
\begin{itemize}
\item $a_0=1$ is used in an addition of the block.

\item An element of a previous addition block of type 1 is used in an addition of the block.
\end{itemize}
We say that an addition block is of {\bf type 2} otherwise.
\end{definition}

\begin{remark}
Observe that it follows directly from this definition and from Lemmas \ref{l3} and \ref{l13} that the addition blocks of type 2 are marked.
\end{remark}

We are now ready to complete the proof of Theorem \ref{t1}.

\subsection{Completing the proof of the upper bound}

Let $K_1$ and $K_2$ be the number of blocks of type 1 and type 2, respectively.  Let $D_1$ and $D_2$ stand for the total number of steps in ${\cal D}$ in blocks of type 1 and 2, respectively. Finally, let $T_1$ and $T_2$ be the number of elements of blocks of type 2 that are used in subsequent blocks of type 1 and type 2, respectively.
Since blocks of type 2 are marked, we have
$$
T_1+T_2\ge K_2.
$$
The number of ways to choose the elements added in steps of blocks of type 1 will be less than or equal to
$$
\exp\left(D_1\log m-(K_1+T_1)\log\log m(1+o(1))\right),\qquad m\to\infty.
$$
The number of ways to choose the elements added in steps of blocks of type 2 will be at most
$$
\exp\left(D_2\log m-T_2\log\log m(1+o(1))\right),\qquad m\to\infty.
$$
We get that the total number of ways to choose which integers are added in addition steps will be smaller than or equal to
$$
\exp\left(r\log m-K\log\log m(1+o(1))\right),\qquad m\to \infty,
$$
where $K$ is the total number of addition blocks in the addition chain.
Multiplying  this by the number of ways of choosing the elements in each set and using Lemma \ref{l4a}, we get
$$
F(m,r)\le \exp\left(r\log m+o\left(\frac{m\log\log m}{\log m}\right)\right),\qquad  m\to \infty.
$$
This completes the proof of Theorem \ref{t1}.

\subsection{Proof of the lower bound}

For the lower bound, we will provide a constructive proof. In this proof, we will make use of binary expansion.

Explicitly, the constructive process goes as follows.

\begin{itemize}
\item {\bf Step 1.} Choose postive integers  $k$ and $u$. We will later describe how to specify these integers.  For now,  assume that $ku\le m$ and that $\nobreak{2^u+k-u\le r}$.

\item {\bf Step 2.} Choose nonnegative integers $s_1,\ldots, s_{k-1}$ such that
$$
\sum_{j=1}^{k-1} s_j +ku=m.
$$

\item {\bf Step 3.} Define ${\cal U}$ as the set of odd positive integers in the interval $[2^{u-1},2^u]$.  Choose $k$ integers $U_1,\ldots, U_{k}$ in the set ${\cal U}$, possibly with repetitions.

\item {\bf Step 4.} Let the first steps of the chain be
$$
a_0=1,\, a_1=2,\, a_2=3,\,a_3=4,\ldots, a_{2^u-2}=2^u-1.
$$

\item {\bf Step 5.} Construct the rest of the chain as follows (using binary expansion):
\begin{align*}
&a_{2^u-1}=U_1,\\
&\vdots\\
&a_{2^u-1+s_1}=U_1\underbrace{0\ldots 0}_{s_1\, \mathrm{zeros}},\\
&\vdots\\
&a_{2^u-1+s_1+u}=U_1\underbrace{0\ldots 0}_{s_1\, \mathrm{zeros}}\underbrace{0\ldots 0}_{u\, \mathrm{zeros}},\\
&a_{2^u+s_1+u}=U_1\underbrace{0\ldots 0}_{s_1\, \mathrm{zeros}}U_2, \\
&\vdots\\
&N=a_{t}=U_1\underbrace{0\ldots 0}_{s_1\, \mathrm{zeros}}U_2 \underbrace{0\ldots 0}_{s_2\, \mathrm{zeros}}U_3\ldots U_{k-1}\underbrace{0\ldots 0}_{s_{k-1}\, \mathrm{zeros}} U_{k},
\end{align*}
\end{itemize}
with
$$
t=2^u+\sum_{j=1}^{k-1}s_j+(k-1)u+k-2.
$$

We can make the following remarks about the chains generated in such a manner.

\begin{itemize}

\item {\bf Remark 1.} The number of binary digits of $N$ is
$$
\sum_{j=1}^{k-1}s_j+ku.
$$

\item {\bf Remark 2.} In light of Step 2, the length of such an addition chain is
$$
2^u+\sum_{j=1}^{k-1}s_j+(k-1)u+k-2=2^u+m-u+k-2< m+r.
$$
Here, the inequality comes from our assumption that $2^u+k-u\le r$.

\item {\bf Remark 3.} Every distinct chain in this family of addition chains yields a different integer.

\item  {\bf Remark 4.} The number of ways of choosing the integers $U_1,\ldots, U_k$ is
$$
\left(\frac{2^u}{4}\right)^{k}.
$$

\item {\bf Remark 5.} The number of ways of choosing nonnegative integers $s_1,\ldots s_{k-1}$ such that
$$
s_1+\cdots+s_{k-1}=m-ku
$$
is equal to
$$
{m-ku+k-2\choose k-1}.
$$

\end{itemize}

From Remarks 1, 2, and 3, every chain constructed is valid and leads to a distinct integer.   From Remarks 4 and 5, we can conclude
\begin{equation}\label{novel1}
F\left(m,r\right)\ge \left(\frac{2^u}{4}\right)^{k}{m-ku+k-2\choose k-1}\ge \left(\frac{2^u(m-ku+k-2)}{4(k-1)}\right)^k.
\end{equation}
To continue, we will assume $r=cm/\log m$.  We will also assume
$2^u+k=r$, that is, $u=\log(r-k)/\log 2$.

Observe that since $k\le r$ and $u\le \log r/\log 2$, we have for $m$ large enough
\begin{equation}\label{novel2}
ku\le \frac{r\log r}{\log 2}\le \frac{cm}{\log 2}.
\end{equation}
Using (\ref{novel2}) in (\ref{novel1}), we get
$$
F\left(m,\frac{cm}{\log m}\right)\ge \left(\frac{2^u m}{k}\right)^k\left(\frac{1}{4}\left(1-\frac{c}{\log 2}\right)\right)^k.
$$
Under the assumptions that $c>0$ and $k=O(m/\log m)$, we deduce from this
$$
F\left(m,\frac{cm}{\log m}\right)\ge 2^{uk}\left(\frac{m}{k}\right)^k\exp\left(O\left(\frac{m}{\log m}\right)\right).
$$
Given that $k\leq r= cm/\log m$, we have
$$
F\left(m,\frac{cm}{\log m}\right)\ge 2^{uk}\exp(k\log\log m)\exp\left(O\left(\frac{m}{\log m}\right)\right).
$$
Using $u=\log(r-k)/\log 2$, it follows that
$$
F\left(m,\frac{cm}{\log m}\right)\ge \exp\left(k\log(r-k)+k\log\log m+O\left(\frac{m}{\log m}\right)\right).
$$
It remains to choose the value of $k$ to make the right-hand side as large as possible.

We set $r-k=y$ and obtain
\begin{eqnarray*}
F\left(m,\frac{cm}{\log m}\right)&\ge& \exp\left((r-y)\log y+(r-y)\log\log m+O\left(\frac{m}{\log m}\right)\right)\\
&=&\exp\left((r-y)\log y+r\log\log m+O(y\log\log m)+O\left(\frac{m}{\log m}\right)\right).
\end{eqnarray*}
The  value of $y$ maximizing $(r-y)\log y$ satisfies
$$
\frac{r}{y}=\log y+1,
$$
This is a transcendental equation and a good approximation of its solution is $y=r/\log r$.  For the sake of simplicity, we use the following approximation of $r/\log r$:
$$
y=\frac{cm}{(\log m)^2}.
$$
We can conclude
\begin{eqnarray*}
& &F\left(m,\frac{cm}{\log m}\right)\\
&\ge& \exp\left(\left(\frac{cm}{\log m}-\frac{cm}{(\log m)^2}\right)(\log m-2\log\log m)+\frac{cm\log\log m}{\log m}+O\left(\frac{m}{\log m}\right)\right)\\
&=&\exp\left(cm-\frac{cm\log\log m}{\log m}+O\left(\frac{m}{\log m}\right)\right).
\end{eqnarray*}
Observe that for any fixed $\varepsilon>0$ and for $m$ large enough,
$$
\exp\left(cm-\frac{cm\log\log m}{\log m}+O\left(\frac{m}{\log m}\right)\right)> \exp\left(cm-\frac{cm\log\log m}{\log m}-\varepsilon\frac{m\log\log m}{\log m}\right).
$$
This completes the proof of the lower bound in Theorem \ref{t1}.

\section{Conclusion}
It would be nice to find an argument that would allow one to close the gap between our lower and upper bounds.  However, the approach used in the proof of the lower bound (say with other values for $k$ and $u$ for instance) cannot yield a much better result.  This is due to the limited number of choices we have in this approach for each additive step.  Indeed, in each additive step, the number of ways of choosing the added integer is $2^u=cm/(\log m)^2$.

One could also consider the following approach to obtain a greater number of chains:
\begin{enumerate}
\item Construct an addition chain with $r$ additive steps of the form $a_j=a_{j-1}+a_t$ with $t<j-1$ and $m$ doubling steps.

\item Choose the position of the additive steps. The number of ways to make this choice is
$$
{ r+m \choose  r }=\exp\left((1+o(1))\frac{cm\log\log m}{\log m}\right).
$$

\item At each addition step, add any previous odd integer.  The number of ways to make this choice is $ r!$.

\end{enumerate}
The number of such chains is
$$
r!{ r+m \choose  r}=\exp\left(cm+o\left(\frac{m \log\log m}{\log m}\right)\right).
$$
The main difficulty with this approach is that some of these chains may lead to the same value. It is not clear how many of these chains lead to different integers. Answering this question would allow one to close the gap between the lower and upper bounds.

\vskip 20pt

\noindent
D\'epartement de math\'ematiques et de statistique, Universit\'e Laval, 1045 Avenue de la m\'edecine, Qu\'ebec, Qu\'ebec G1V0A6, Canada \\
Email address: jmdk@mat.ulaval.ca
\vskip 7pt

\noindent
D\'epartement de math\'ematiques et de statistique, Universit\'e Laval, 1045 Avenue de la m\'edecine, Qu\'ebec, Qu\'ebec G1V0A6, Canada \\
Email address: nicolas.doyon@mat.ulaval.ca
\vskip 7pt

\noindent
Department of Mathematics, University of Toronto, 40 St George Street, Toronto, Ontario M5S2E4, Canada \\
Email address: william.verreault@utoronto.ca

\end{document}